\documentclass[]{article}
\usepackage{graphicx}
\usepackage{epstopdf}
\usepackage{amsfonts}
\usepackage{amsmath}
\usepackage{amssymb}
\usepackage{fancyhdr}
\usepackage{titlesec}
\usepackage{indentfirst}
\usepackage{booktabs}
\usepackage{verbatim}
\usepackage{color}
\usepackage{amsthm}
\usepackage{subfigure}

\usepackage[page,header]{appendix}
\usepackage{titletoc}

\numberwithin{equation}{section}
\newtheorem{theorem}{Theorem}[section]
\newtheorem{lemma}[theorem]{Lemma}

\newtheorem{proposition}[theorem]{Proposition}

\newtheorem{remark}[theorem]{Remark}

\newcommand{\rd}{\mathrm{d}}

\newcommand{\im}{\mathrm{i}}
\newcommand{\e}{\mathrm{e}}
\newcommand{\mQ}{\mathcal{Q}}
\newcommand{\mP}{\mathcal{P}}

\newcommand{\bR}{\mathbb{R}}
\newcommand{\bS}{\mathbb{S}}

\newcommand{\Sd}{\mathbb{S}}
\newcommand{\D}{\mathcal{D}_L}

\newcommand{\fe}{f_{N}}

\begin{document}

	\title{A new stability and convergence proof \\of the Fourier-Galerkin spectral method \\for the spatially homogeneous Boltzmann equation}

	\author{Jingwei Hu\footnote{Department of Mathematics, Purdue University, West Lafayette, IN 47907, USA (jingweihu@purdue.edu).}, \  
	Kunlun Qi\footnote{Department of Mathematics, City University of Hong Kong, Hong Kong, China (kunlun.qi@my.cityu.edu.hk).}, \  
	and \ Tong Yang\footnote{Department of Mathematics, City University of Hong Kong, Hong Kong, China (matyang@cityu.edu.hk).}}    
	\maketitle
	
\begin{abstract}
Numerical approximation of the Boltzmann equation is a challenging problem due to its high-dimensional, nonlocal, and nonlinear collision integral. Over the past decade, the Fourier-Galerkin spectral method \cite{PR00} has become a popular deterministic method for solving the Boltzmann equation, manifested by its high accuracy and potential of being further accelerated by the fast Fourier transform. Albeit its practical success, the stability of the method is only recently proved in \cite{FM11} by utilizing the ``spreading" property of the collision operator. In this work, we provide a new proof based on a careful $L^2$ estimate of the negative part of the solution. We also discuss the applicability of the result to various initial data, including both continuous and discontinuous functions.
\end{abstract}

{\small 

}

%\tableofcontents

%%%%%%%%%%%%%%%%%%%%%%%%%%%%%%%%%%%%%%	
\section{Introduction}
\label{sec:intro}

The Boltzmann equation is one of the fundamental equations in kinetic theory and serves as a basic building block to connect microscopic Newtonian mechanics and macroscopic continuum mechanics \cite{Cercignani, Villani02}. Albeit its wide applicability, numerical approximation of the Boltzmann equation is a challenging scientific problem due to the complicated structure of the equation (high-dimensional, nonlinear, and nonlocal). As such, the particle based direct simulation Monte Carlo method (DSMC) \cite{Bird} has been widely used in various applications for its simplicity and low computational cost. Nevertheless, the stochastic method suffers from slow convergence and becomes extremely expensive when simulating non-steady and low-speed flows. 

Since the pioneering work \cite{PP96, PR00}, it has been realized that the Fourier-Galerkin spectral method offers a suitable framework to approximate the Boltzmann collision operator. First of all, it is a deterministic method and provides very accurate results compared with stochastic method. Secondly, the Boltzmann collision operator is translation-invariant and the Fourier basis exactly leverages this structure. Thirdly, after the Galerkin projection, the collision operator presents a convolution-like structure, which opens the possibility to further accelerate the method by the fast Fourier transform (FFT) \cite{MP06, GHHH17}. Because of the above reasons, over the past decade, the Fourier spectral method has become a very popular deterministic method for solving the Boltzmann equation and related collisional kinetic models, see for instance, \cite{PRT00, FR03, FMP06, HY12, JAH19_1}, or the recent review article \cite{pareschi}.

As opposed to its practical success, the theoretical study of the Fourier spectral method is quite limited, largely because the spectral approximation destroys the positivity of the solution, yet the positivity is one of the key properties to study the well-posedness of the equation. In \cite{PR00stability}, a positivity-preserving filter is applied to the equation to enforce the positivity of the solution. As a result, the stability of the method can be easily proved. However, the filter often comes with the price of significantly smearing the solution (hence destroying the spectral accuracy) and should be used only when the solution contains discontinuities (to suppress the oscillations caused by Gibbs phenomenon). Recently, a stability proof for the original Fourier spectral method is established in \cite{FM11}, where the authors provide a quite complete study of the method including both finite and long time behavior. The key strategy in \cite{FM11} is to use the ``spreading" or ``mixing" property of the collision operator to show that the solution will become everywhere positive after a small time. Motivated by this work, we present in this paper a different well-posedness and stability proof. The main difference from \cite{FM11} lies in that, instead of requiring the solution to be positive everywhere which is a stronger condition to achieve, we show that the $L^2$ norm of the negative part of the solution can be controlled as long as it is small initially. In other words, the solution is allowed to be negative for the method to remain stable. Therefore, our strategy does not rely on any sophisticated property of the collision operator and provides a simpler proof. In addition, we quantify clearly the requirement on the initial condition for the method to be stable, which includes both continuous and discontinuous functions. 

We mention another line of research which develops the conservative-spectral approximation for the Boltzmann equation \cite{GT09}. Apart from apparent differences (the Fourier-Galerkin method considered in this paper is based on domain truncation and periodization, while the method \cite{GT09} is based on Fourier transform and no periodization is performed), a conservation subroutine is added to restore the mass, momentum, and energy conservation. As a consequence, the method is able to preserve the Maxwellian distribution as time goes to infinity. The stability and convergence of the method is recently established in \cite{AGT18}, where the Fourier projection is only applied to the gain part of the collision operator. In contrast, both gain and loss terms are projected in our method, hence the loss term does not possess a definite sign.

The paper is essentially self-contained. In Section~\ref{sec:review}, we briefly review the Fourier-Galerkin spectral method for the spatially homogeneous Boltzmann equation. After that, we discuss the basic assumptions (e.g., the collision kernel and truncation parameters) used throughout the paper. The assumptions on the initial condition are addressed in Section~\ref{subsec:initial}, which will play an important role in proving the main result. In Section~\ref{sec:QR} (and Appendix), we provide some preliminary estimates on the truncated collision operator. These are known results in the whole space but some subtle differences appear in the torus. Section~\ref{sec:main} presents our main result. We first conduct a $L^2$ estimate of the negative part of the solution and then prove a local existence/uniqueness result. Finally, the well-posedness and stability of the method on an arbitrary bounded time interval is established in Section~\ref{subsec:main} (Theorem~\ref{existencetheorem}). Facilitated with the stability result, the paper is concluded in Section~\ref{sec:conv} with a straightforward convergence and spectral accuracy proof of the method.

%%%%%%%%%%%%%%%%%%%%%%%%%%%%%%%%%%%%%%	
\section{Fourier-Galerkin spectral method for the spatially homogeneous Boltzmann equation}
\label{sec:review}

In this section, we review the Fourier-Galerkin spectral method for the spatially homogeneous Boltzmann equation. The presentation follows the formulation originally proposed in \cite{PR00} which is the basis for many fast algorithms developed recently \cite{GHHH17, HM19, HQ20}. Here we limit the description to the extent that is sufficient for the following proof. At the end of the section, we discuss the basic assumptions used throughout the rest of the paper, in particular, the assumptions on the initial condition.

The spatially homogeneous Boltzmann equation reads
\begin{equation} \label{BE}
\partial_{t} f= Q(f,f), \quad t>0, \ v\in \mathbb{R}^d, \ d\geq 2,
\end{equation}
where $f=f(t,v)$ is the probability density function of time $t$ and velocity $v$, $Q$ is the collision operator describing the binary collisions among particles, whose bilinear form is given by
\begin{equation} \label{Qstrong}
Q(g,f)(v)=\int_{\bR^d}\int_{\Sd^{d-1}}B(|v-v_*|,\cos \theta)[g(v_*')f(v')-g(v_*)f(v)]\,\rd{\sigma}\, \rd{v_*}.
\end{equation}
In (\ref{Qstrong}), $\sigma$ is a vector varying over the unit sphere $\Sd^{d-1}$, $v'$ and $v_*'$ are defined as
\begin{equation}
v'=\frac{v+v_*}{2}+\frac{|v-v_*|}{2}\sigma, \quad v_*'=\frac{v+v_*}{2}-\frac{|v-v_*|}{2}\sigma,
\end{equation}
and $B\geq 0$ is the collision kernel. In this paper we will consider the kernel of the form
\begin{equation} \label{kernel}
B(|v-v_*|,\cos \theta)=\Phi(|v-v_*|) b(\cos \theta), \quad \cos \theta =\frac{\sigma\cdot (v-v_*)}{|v-v_*|},
\end{equation}
whose kinetic part $\Phi$ is a non-negative function and angular part $ b$ satisfies the Grad's cut-off assumption
\begin{equation} \label{cutoff}
\int_{\Sd^{d-1}} b(\cos \theta)\,\rd{\sigma}<\infty.
\end{equation}

To apply the Fourier-Galerkin spectral method, we consider an approximated problem of (\ref{BE}) on a torus $\mathcal{D}_L=[-L,L]^d$:
\begin{equation} \label{ABE}
\left\{
\begin{split}
&\partial_{t} f = Q^{R}(f,f), \quad t>0, \ v\in \mathcal{D}_L,\\
& f(0,v)= f^{0}(v), 
\end{split}
\right.
\end{equation}
where the initial condition $f^0$ is a non-negative periodic function, $Q^{R}$ is the truncated collision operator defined by
\begin{equation}\label{QR}
\begin{split}
Q^R(g,f)(v)&=\int_{\mathcal{B}_R}\int_{\Sd^{d-1}}\Phi(|q|)b(\sigma\cdot \hat{q})\left[g(v_*')f(v')-g(v-q)f(v)\right]\, \rd{\sigma}\, \rd{q}\\
&=\int_{\bR^d} \int_{\Sd^{d-1}}\mathbf{1}_{|q|\leq R}\Phi(|q|)b(\sigma\cdot \hat{q})\left[g(v_*')f(v')-g(v-q)f(v)\right]\,\rd{\sigma}\, \rd{q},
\end{split}
\end{equation}
where a change of variable $v_*\rightarrow q=v-v_*$ is applied and the new variable $q$ is truncated to a ball $\mathcal{B}_R$ with radius $R$ centered at the origin. We write $q=|q|\hat{q}$ with $|q|$ being the magnitude and $\hat{q}$ being the direction. Accordingly, 
\begin{equation}
v'=v-\frac{q-|q|\sigma}{2}, \quad v_*'=v-\frac{q+|q|\sigma}{2}.
\end{equation}
In practice, the values of $L$ and $R$ are often chosen by an anti-aliasing argument \cite{PR00}: assume that $\text{Supp} (f^0(v))\subset \mathcal{B}_S$, then one can take
\begin{equation} \label{RL1}
R=2S, \quad L\geq \frac{3+\sqrt{2}}{2}S.
\end{equation}

Given an integer $N\geq0$, we then seek a truncated Fourier series expansion of $f$ as
\begin{equation}
f(t,v)\approx f_N(t,v)=\sum\limits_{k = -N/2}^{N/2} f_k(t) \e^{\im \frac{\pi}{L}k\cdot v} \in \mathbb{P}_N,
\end{equation}
where 
\begin{equation}
\mathbb{P}_N=\text{span} \left\{ \e^{\im \frac{\pi}{L} k\cdot v}\Big| -N/2\leq k \leq N/2 \right\}\footnote{Note here $k=(k_1,\dots,k_d)$ is a vector, $-N/2\leq k \leq N/2$ means $-N/2\leq k_j \leq N/2$, $j=1,\dots,d$, and $\sum_{k=-N/2}^{N/2}:=\sum_{k_1=-N/2}^{N/2}\cdots \sum_{k_d=-N/2}^{N/2}$.},
\end{equation}
equipped with inner product 
\begin{equation}
\langle f,g \rangle = \frac{1}{(2L)^{d}}\int_{\D} f \bar{g}\, \rd v.
\end{equation}
Substituting $f_N$ into (\ref{ABE}) and conducting the Galerkin projection onto the space $\mathbb{P}_N$ yields
\begin{equation} \label{PFS}
\left\{
\begin{split}
&\partial_{t} f_N = \mP_N Q^{R}(f_N,f_N), \quad t>0, \ v\in \mathcal{D}_L,\\
& f_N(0,v)=f_{N}^{0}(v),
\end{split}
\right.
\end{equation}
where $\mP_N$ is the projection operator: for any function $g$,
\begin{equation}\label{proj}
\mP_N g=\sum_{k=-N/2}^{N/2}\hat{g}_k \e^{\im \frac{\pi}{L}k\cdot v}, \quad \hat{g}_k=\langle g, \e^{\im \frac{\pi}{L}k\cdot v}\rangle,
\end{equation}
$f_N^0\in \mathbb{P}_N$ is the initial condition to the numerical system and should be a reasonable approximation to $f^0$. More discussion on the initial condition will be given in Section~\ref{subsec:initial}, which in fact plays an important role in the following proof.

Writing out each Fourier mode of (\ref{PFS}), we obtain
\begin{equation} \label{FS}
\left\{
\begin{split}
&\partial_{t} f_k = Q^{R}_k, \quad  -N/2\leq k\leq N/2,\\
& f_k(0)= f^{0}_k,
\end{split}
\right.
\end{equation}
with
\begin{equation}
Q_{k}^R:=\langle Q^R(f_N,f_N), \e^{\im \frac{\pi}{L}k\cdot v}\rangle, \quad f^0_k:=\langle f_N^0, \e^{\im \frac{\pi}{L}k\cdot v}\rangle.
\end{equation}
Using the definition in (\ref{QR}) and orthogonality of the Fourier basis, we can derive that
\begin{equation} \label{sum}
Q_{k}^R =\sum\limits_{\substack{l,m=-N/2\\l+m=k}}^{N/2} G(l,m)f_lf_m,
\end{equation}
where the weight $G$ is given by
\begin{equation}
\begin{split}
G(l,m) &= \int_{\mathcal{B}_{R}}\int_{\Sd^{d-1}}\Phi(|q|)b(\sigma\cdot \hat{q})\left[ \e^{-\im \frac{\pi}{2L}(l+m)\cdot q +\im \frac{\pi}{2L}|q|(l-m)\cdot \sigma} - \e^{-\im \frac{\pi}{L}m\cdot q} \right]\,\rd\sigma\,\rd q\\
&=  \int_{\mathcal{B}_{R}}\e^{-\im \frac{\pi}{L}m\cdot q}\left[\int_{\Sd^{d-1}}\Phi(|q|)b(\sigma\cdot \hat{q})(\e^{\im \frac{\pi}{2L}(l+m)\cdot (q-|q|\sigma)}-1)\, \rd\sigma\right] \,\rd q. \label{GG}
\end{split}
\end{equation}
The second equality above is obtained by switching two variables $\sigma \leftrightarrow \hat{q}$ in the gain part of $G(l,m)$. In the direct Fourier spectral method, $G(l,m)$ is precomputed since it is independent of the solution. Then in the online computation, the sum (\ref{sum}) is evaluated directly. 

Note that the solution $f$ to the original problem (\ref{ABE}) is always non-negative which is the key to many stability estimates. However, the solution $f_N$ to the numerical system (\ref{PFS}) is not necessarily non-negative due to the spectral projection which constitutes the main difficulty in the numerical analysis. Luckily, by virtue of the Fourier spectral method, mass is always conserved which provides some control of the solution. Precisely, we have
\begin{lemma}  \label{lemma:conv}
The numerical system (\ref{PFS}) preserves mass, that is,
\begin{equation} 
\int_{\D} f_N(t,v) \,\rd v=\int_{\D} f^{0}_N(v) \,\rd v.
\end{equation}
\end{lemma}

\begin{proof}
Note that 
\begin{equation} 
\int_{\mathcal{D}_L} f_N(t,v)\,\rd{v}=\sum_{k=-N/2}^{N/2} f_k(t) \int_{\mathcal{D}_L}  \e^{\im \frac{\pi}{L}k\cdot v}\, \rd{v}=(2L)^d f_0(t),
\end{equation}
where $f_0$ is the zero-th mode of the numerical solution and is governed by
\begin{equation} 
\partial_{t} f_0 = Q_0^R.
\end{equation}
From (\ref{sum}), it is clear that $Q^R_0\equiv0$ since $G(l,m)\equiv0$ when $l+m=0$. This implies $f_0$ remains constant in time, whose value is the zero-th Fourier mode of the initial condition $f_N^0(v)$.
\end{proof}

We now introduce some assumptions and notations that will be used throughout the rest of this paper.

\vspace{0.05in}
\noindent{\bf Basic assumptions on the truncation parameters and the collision kernel.} 

\begin{itemize}
\vspace{-0.05in}
\item[(1)] The truncation parameters $L$ and $R$ in (\ref{ABE}) satisfy 
\begin{equation} \label{RL}
L\geq R>0.
\end{equation}
Note that the choice (\ref{RL1}) implies $L\geq (3+\sqrt{2})R/4$ hence the above condition is satisfied.

\vspace{-0.1in}
\item[(2)] The kinetic part of the collision kernel (\ref{kernel}) satisfies
\begin{equation} \label{kinetic}
\left \| \mathbf{1}_{|v|\leq R}\Phi(|v|)\right\|_{L^{\infty}(\D)} < \infty.
\end{equation}
Note that all power law hard potentials $\Phi(|v|)=|v|^{\gamma}$ ($0\leq \gamma\leq 1$) as well as the ``modified" soft potentials $\Phi(|v|)=(1+|v|)^{\gamma}$ ($-d<\gamma <0$) satisfy this condition. 

\vspace{-0.1in}
\item[(3)] The angular part of the collision kernel (\ref{kernel}) has been replaced by its symmetrized version\footnote{This symmetrization can readily reduce the computational cost by a half (integration over the whole sphere is reduced to half sphere) so it also has important implications for numerical purpose, see \cite{GHHH17}.}:
\begin{equation} \label{angular}
\left[b(\cos \theta)+b(\cos \left(\pi-\theta\right))\right]\mathbf{1}_{0 \leq \theta \leq \pi/2},
\end{equation}
and satisfies the cut-off assumption (\ref{cutoff}).
\end{itemize}

\noindent{\bf Some notations.} 

For a periodic function $f(v)$ in $\D$, we define its Lebesgue norm and Sobolev norm as follows:
\begin{equation}
\|f\|_{L^p_{\text{per}}(\D)}=\left(\int_{\D} |f(v)|^p\,\rd{v}\right)^{1/p}, \quad \|f\|_{{H^k_{\text{per}}}(\D)}=\left( \sum_{|\nu|\leq k} \|\partial_v^{\nu}f\|_{L^2_{\text{per}(\D)}}^2\right)^{1/2},
\end{equation}
where $k\geq 0$ is an integer and $\nu$ is a multi-index. ``per" indicates the function is periodic and will not be included in the following for simplicity. 

Except in Section~\ref{sec:QR}, we do not track explicitly the dependence of constants on the truncation parameters $R$, $L$, dimension $d$, and the collision kernel $B$.

For a function $f(v)$ in $\D$, we define its positive and negative parts as 
\begin{equation}
f^+(v)=\max\limits_{v\in\D}\{ f(v),0 \}, \quad f^-(v)=\max\limits_{v\in\D}\{ -f(v),0 \},
\end{equation}
so that $f=f^+-f^-$ and $|f|=f^++f^-$.

%%%%%%%%%%%%%%%%%%%%%%%%%%%%%%%%%%%%%%%
\subsection{Assumptions on the initial condition}
\label{subsec:initial}

To prove our main well-posedness and stability result, Theorem~\ref{existencetheorem}, we would assume that the initial condition $f^0(v)$ to the original problem (\ref{ABE}) is periodic, non-negative, and belongs to $L^1\cap H^1(\D)$ (in fact $L^1$ can be removed since $L^2(\D)\subset L^1(\D)$ due to boundedness of the domain). For the initial condition $f_N^0(v)$ to the numerical system (\ref{PFS}), we would require it to lie in the space $\mathbb{P}_N$ and satisfies the following:
\begin{itemize}
\item[(a)] Mass conservation:
\begin{equation} \label{con(a)}
\int_{\D} f^{0}_N(v)\, \rd v=\int_{\D} f^0(v)\, \rd v.
\end{equation}
\item[(b)] Control of $L^2$ and $H^1$ norms: for any integer $N\geq 0$,
\begin{equation} \label{con(b)}
\|f^0_N\|_{L^2(\D)}\leq \|f^0\|_{L^2(\D)}, \quad \|f^0_N\|_{H^1(\D)}\leq \|f^0\|_{H^1(\D)}.
\end{equation}
\item[(c)] Control of $L^1$ norm: there exists an integer $N_0$ such that for all $N> N_0$,
\begin{equation} \label{con(c)}
\|f_N^0\|_{L^1(\D)}\leq C \|f^0\|_{L^1(\D)}.
\end{equation}
where $C>1$ is some constant whose value is of no essential importance. In the following proof, we will take $C=2$ for simplicity.
\item[(d)] $L^2$ norm of $f_N^{0,-}$ can be made arbitrarily small: for any $\varepsilon>0$, there exists an integer $N_0$ such that for all $N> N_0$,
\begin{equation} \label{con(d)}
\|f_N^{0,-}\|_{L^2(\D)} <\varepsilon.
\end{equation}
\end{itemize}

\begin{remark} \label{rmk}
An obvious choice is to take $f_N^0=\mP_N f^0$. Condition (a) is satisfied since it is equivalent to preserving the zero-th Fourier mode of the function. Condition (b) is a direct consequence of the Parseval's identity. Condition (c) can be obtained by the $L^2$ convergence of the Fourier series and that $L^1$ norm can be controlled by $L^2$ norm. Condition (d) can be proved at least when the uniform convergence of the Fourier series is guaranteed, for which one may require additional continuity on $f^0$. For instance, $f^0$ is H\"{o}lder continuous, or continuous plus bounded variation (in fact $BV$ can be removed since $H^1(\D)\subset W^{1,1}(\D)\subset BV(\D)$).
\end{remark}

\begin{remark}
Sometimes the initial condition $f^0$ may contain discontinuities, then simply taking the Fourier projection of $f^0$ will generate undesirable oscillations (Gibbs phenomenon). Hence a reasonable choice is to take a filtered version $f_N^0=\mathcal{S}_Nf^0$, where $\mathcal{S}_N$ is defined as: for any function g,
\begin{equation}
\mathcal{S}_Ng=\sum_{k=-N/2}^{N/2}\sigma_N(k)\hat{g}_k\e^{\im \frac{\pi}{L}k\cdot v},\quad \hat{g}_k=\langle g, \e^{\im \frac{\pi}{L}k\cdot v}\rangle,
\end{equation}
with $\sigma_N$ being the filter function, see for instance \cite[Chapter 9]{HGG07}. Typically, the filter won't change the zero-th Fourier mode of the function, and won't amplify the remaining Fourier modes, hence conditions (a) and (b) would be satisfied automatically. For conditions (c) and (d) to hold, one needs some kind of convergence which depends on the property of the actual filter. Without going into details, let us just mention that there is a class of positive filters (e.g., the Fej\'{e}r or Jackson filter \cite{WWAF06}) which can preserve the positivity of the function so that the condition (d) is trivially satisfied. Condition (c) can be satisfied as well by using the Young's inequality and the $L^1$ norm of the filter is exactly 1. However, the positivity-preserving filters may come with the price of slower convergence (away from the discontinuity) compared with other high order filters (e.g., the exponential filter \cite{HGG07}). Therefore, one could take non-positive high order filters, as long as they satisfy the conditions (c) and (d). It is worth emphasizing that the purpose of applying the filter here is merely to fix the initial condition when $f^0$ is discontinuous so that our well-posedness and stability proof still holds. This is in stark contrast to the filtering method used in \cite{PR00stability} and \cite{CFY18}, where the filter is applied to the equation to preserve the positivity of the solution. 
\end{remark}

%%%%%%%%%%%%%%%%%%%%%%%%%%%%%%%%%%%%%%%
\section{Some preliminary estimates on the truncated collision operator $ Q^{R}$}
\label{sec:QR}

In this section, we prove some important estimates for the truncated collision operator (\ref{QR}). Since its gain term and loss term possess quite different properties, we consider 
\begin{equation}
\begin{split}
Q^{R,+}(g,f)(v)&:=\int_{\bR^d} \int_{\Sd^{d-1}}\mathbf{1}_{|q|\leq R}\Phi(|q|)b(\sigma\cdot \hat{q})g(v_*')f(v')\,  \rd{\sigma} \,\rd{q},\\
Q^{R,-}(g,f)(v)&:=\int_{\bR^d} \int_{\Sd^{d-1}}\mathbf{1}_{|q|\leq R}\Phi(|q|)b(\sigma\cdot \hat{q})g(v-q)f(v) \,\rd{\sigma}\, \rd{q},
\end{split}
\end{equation}
separately whenever appropriate. 

\begin{proposition}
	Let the collision kernel $ B $ and truncation parameters $R$ and $L$ satisfy the assumptions (\ref{RL}), (\ref{kinetic}), (\ref{angular}), and (\ref{cutoff}), then the truncated collision operators $ Q^{R,\pm}(g,f)$ satisfy the following estimates: for $1\leq p \leq \infty$,	
	\begin{equation}\label{QGLp1}
	\left\|Q^{R,+}(g,f)\right\|_{L^{p}(\D)} \leq C^+_{R,L,d,p}(B) \left\|g\right\|_{L^{1}(\D)} \left\|f \right\|_{L^{p}(\D)},
	\end{equation}
	where the constant $C^+_{R,L,d,p}(B)=C^{1/p}\|b\|_{L^1(\mathbb{S}^{d-1})}\|\mathbf{1}_{|v|\leq R} \Phi(|v|)\|_{L^{\infty}{(\D)}}$.
	\begin{equation}\label{QLLp}
	\left\|Q^{R,-}(g,f)\right\|_{L^{p}(\D)} \leq C_{R,L,d}^-(B) \left\|g\right\|_{L^{1}(\D)} \left\|f \right\|_{L^{p}(\D)},
	\end{equation}
	where the constant $C^-_{R,L,d}(B)=C\|b\|_{L^1(\mathbb{S}^{d-1})} \left \| \mathbf{1}_{|v|\leq R}\Phi(|v|)\right\|_{L^{\infty}(\D)}$.
	
	In particular, for the whole collision operator $ Q^{R}(g,f)$, we have
	\begin{equation}\label{QLp}
	\left\|Q^{R}(g,f)\right\|_{L^{p}(\D)} \leq C_{R,L,d,p}(B) \left\|g\right\|_{L^{1}(\D)} \left\|f \right\|_{L^{p}(\D)}.
	\end{equation}	
\end{proposition}

\begin{proof}
The proof of the truncated gain term $ \mQ^{R,+}(g,f) $ is similar to the usual Boltzmann operator $ \mQ^+(g,f)$ on $\mathbb{R}^d$. However, the right hand side is not entirely obvious as we need to restrict back to a bounded domain. Therefore, we follow \cite[Theorem 2.1]{mouhot2004regularity} to give a complete proof of (\ref{QGLp1}) (see Appendix). In fact, by carrying out this carefully, one can see that the condition (\ref{RL}) is needed.

For the loss term, we write it as
\begin{equation} \label{QLS}
Q^{R,-}(g,f)(v)=L^R(g)(v)f(v),
\end{equation}
where $L^R$ is a convolution given by
\begin{equation}
L^R(g)(v)=\|b\|_{L^1(\mathbb{S}^{d-1})} \int_{\bR^d} \mathbf{1}_{|q|\leq R}\Phi(|q|) g(v-q)\,\rd{q}=\|b\|_{L^1(\mathbb{S}^{d-1})}\left(\mathbf{1}_{|v|\leq R}\Phi(|v|)\right)*g(v).
\end{equation}
Then
\begin{equation}
	\begin{split}
	\|Q^{R,-}(g,f)\|_{L^p(\D)}& \leq \left \| L^R(g)\right\|_{L^{\infty}(\D)}\|f\|_{L^p(\D)}\\
	& = \|b\|_{L^1(\mathbb{S}^{d-1})} \left \| \left(\mathbf{1}_{|v|\leq R}\Phi(|v|)\right)*g(v)\right\|_{L^{\infty}(\D)}\|f\|_{L^p(\D)}\\
	& \leq \|b\|_{L^1(\mathbb{S}^{d-1})} \left \| \mathbf{1}_{|v|\leq R}\Phi(|v|)\right\|_{L^{\infty}(\D)} \left\|g\right\|_{L^{1}(\mathcal{B}_{\sqrt{2}L+R})} \|f\|_{L^p(\D)}\\
	& \leq C \|b\|_{L^1(\mathbb{S}^{d-1})} \left \| \mathbf{1}_{|v|\leq R}\Phi(|v|)\right\|_{L^{\infty}(\D)}\|g\|_{L^1(\D)}\|f\|_{L^p(\D)}\\
	& = C_{R,L,d}^-(B) \left\|g\right\|_{L^{1}(\D)} \left\|f \right\|_{L^{p}(\D)},
	\end{split}
	\end{equation}
where we used $R\leq L$ in the third line and $g$ is a periodic function on $\D$ in the fourth line.
\end{proof}

\begin{proposition}
Let the collision kernel $ B $ and truncation parameters $R$ and $L$ satisfy the assumptions (\ref{RL}), (\ref{kinetic}), (\ref{angular}), and (\ref{cutoff}), then the truncated collision operator $ Q^{R}(g,f)$ satisfies the following estimate: for integer $k\geq 0$,
\begin{equation} \label{QHk}
	\left\|Q^{R}(g,f)\right\|_{H^k(\D)} \leq C'_{R,L,d,k}(B) \left\|g\right\|_{H^k(\D)} \left\|f \right\|_{H^k(\D)}.
\end{equation}
\end{proposition}

\begin{proof}
First of all, (\ref{QHk}) when $k=0$ is a direct consequence of (\ref{QLp}) by taking $p=2$ and noting that $\left\|g\right\|_{L^{1}(\D)}\leq (2L)^{d/2}\left\|g\right\|_{L^{2}(\D)}$.

To prove (\ref{QHk}) for $k>0$, note that the collision operator satisfies the Leibniz rule:
	\begin{equation}
	\partial_v^{\nu}Q^{R}(g,f)=  \sum_{\mu\leq \nu}\binom{\nu}{\mu}  Q^{R}(\partial_v^{\mu}g,\partial_v^{\nu-\mu}f),
	\end{equation}
	which is a consequence of the bilinearity and the Galilean invariance of the truncated collision operator $Q^{R}(g,f)(v-h)=Q^{R}(g(v-h),f(v-h))$. Then we have
	\begin{equation}
	\begin{split}
	\|Q^{R}(g,f)\|_{H^k(\D)}^2&=\sum_{|\nu|\leq k}\left\|\partial^{\nu}_v Q^{R}(g,f)\right\|_{L^2(\D)}^2=\sum_{|\nu|\leq k} \left\|\sum_{\mu\leq \nu}\binom{\nu}{\mu}  Q^{R}(\partial_v^{\mu}g,\partial_v^{\nu-\mu}f)\right\|^2_{L^2(\D)}\\
	& \leq \sum_{|\nu|\leq k}\sum_{\mu\leq \nu}\binom{\nu}{\mu}^2 \sum_{\mu\leq \nu} \left\| Q^{R}(\partial_v^{\mu}g,\partial_v^{\nu-\mu}f)\right\|^2_{L^2(\D)} \\
	& \leq C'^2_{R,L,d,0}(B) \sum_{|\nu|\leq k}\sum_{\mu\leq \nu}\binom{\nu}{\mu}^2  \sum_{\mu\leq \nu}  \left\|\partial_v^{\mu}g\right\|^2_{L^{2}(\D)} \left\| \partial_v^{\nu-\mu}f \right\|^2_{L^{2}(\D)}\\
	& \leq C'^2_{R,L,d,k}(B) \left\|g\right\|^2_{H^{k}(\D)} \left\|f \right\|^2_{H^k(\D)},
	\end{split}
	\end{equation}
where we used the Cauchy-Schwarz inequality in the second line.
\end{proof}

%%%%%%%%%%%%%%%%%%%%%%%%%%%%%%%%%%%%%%%

\section{Main result: well-posedness and stability of the method}
\label{sec:main}

In this section, we establish the well-posedness and stability of the Fourier-Galerkin spectral method (\ref{PFS}) on an arbitrary bounded time interval $[0, T]$. The main strategy of the proof is as follows: In Section~\ref{sec:regularity} we prove some $L^2$ and $H^k$ estimates of the solution under {\it a priori} $L^1$ bound of $f_N$, among which the key result is the $L^2$ estimate of the negative part of the solution (Proposition~\ref{regularity1}). Proposition~\ref{localexistence} is a local existence and uniqueness result over a small time interval $ [t_0, t_0+\tau]$. Finally the main result is presented in Theorem~\ref{existencetheorem}, where we show that when $N$ is large enough the negative part of the solution can be controlled over time $[0,\tau]$. Due to mass conservation, this consequently implies the initial $L^1$ bound of the solution can be restored at time $\tau$. Therefore, we can repeat the procedure iteratively to build the solution up to final time $T$ (the estimates on $N$ and $\tau$ are done carefully at the beginning so that the same values can be used in the following iteration).

\subsection{Propagation of the $L^2$ estimate of $f_N^-$ under {\it a priori} $L^1$ bound of $f_N$}
\label{sec:regularity}

We first establish the $L^2$ and $H^k$ estimates of $f_N$ under {\it a priori} $L^1$ bound of $f_N$. This result is not new and the proof is similar to \cite[Lemma 4.2]{FM11}. The main difference is that we closely track the dependence in the case of $H^1$ which will be useful in the following estimate.

\begin{proposition}\label{regularity}
Let the collision kernel $ B $ and truncation parameters $R$ and $L$ satisfy the assumptions (\ref{RL}), (\ref{kinetic}), (\ref{angular}), and (\ref{cutoff}). For the numerical system \eqref{PFS}, assume that the initial condition $ f_{N}^0\in H^{k}(\D) $ for some integer $k\geq 0$ and that the solution $ \fe $ has a $ L^{1}$ bound up to some time $t_0$:
	\begin{equation}\label{fNL1}
	\forall t\in [0,t_0], \quad \left\| \fe(t) \right\|_{L^{1}(\D)} \leq  M,
	\end{equation}
	then there exists a constant $K_k$ depending on $t_0$, $M$, and $\|f_N^0\|_{H^k(\D)}$ such that  
	\begin{equation} \label{Hk}
	\forall t\in[0,t_0], \quad \left\| \fe(t) \right\|_{H^{k}(\D)} \leq K_{k}\left(t_0,M, \|f_N^0\|_{H^k(\D)}\right).
	\end{equation}
	In particular, for $k=0$ and $k=1$, we have
	\begin{equation} 
	K_0=\e^{t_0 D_0M }  \left\|f_{N}^0\right\|_{L^{2}(\D)}, \quad K_1=\e^{t_0D_1 \left(M+K_0\right)} \left( \left\|f^0_N\right\|_{H^{1}(\D)}+D_2\right), \label{KK1}
	\end{equation}
where $D_0$, $D_1$, $D_2$ are constants depending only on the truncation parameters $R$, $L$, dimension $d$, and the collision kernel $B$.
\end{proposition}

\begin{proof}		
	The proof is based on mathematical induction.
	
	Step (i): We first prove (\ref{Hk}) holds for $k=0$. Multiplying both sides of (\ref{PFS}) by $f_N$ and integrating over $\D$ yields
	\begin{equation}\label{H0}
	\begin{split}
	\frac{1}{2} \frac{\rd}{\rd t} \left\| \fe \right\|^{2}_{L^{2}(\D)} =&  \int_{\D} \mP_NQ^{R}(\fe,\fe) \fe \, \rd v \leq  \left\| \mP_NQ^{R}(\fe,\fe) \right\|_{L^{2}(\D)} \left\| \fe \right\|_{L^{2}(\D)} \\
	\leq  &\left\| Q^{R}(\fe,\fe) \right\|_{L^{2}(\D)} \left\| \fe \right\|_{L^{2}(\D)} 
	\leq  D_0\left\| \fe \right\|_{L^{1}(\D)} \left\| \fe \right\|^{2}_{L^{2}(\D)} 
	\leq  D_0M \left\| \fe \right\|^{2}_{L^{2}(\D)},
        \end{split}
	\end{equation}
	where we used (\ref{QLp}) and the assumption (\ref{fNL1}).
	Thus we have
	\begin{equation}
	\frac{\rd}{\rd t} \left\| \fe \right\|_{L^{2}(\D)} \leq D_0M \left\| \fe \right\|_{L^{2}(\D)}.
	\end{equation}
        By the Gr{\"o}nwall's inequality, we further conclude that 
	\begin{equation}\label{priori}
	 \left\| \fe(t) \right\|_{L^{2}(\D)} \leq \e^{ D_0M t_0}  \left\|f_{N}^0\right\|_{L^{2}(\D)},  \quad \forall t\in [0, t_0].
	\end{equation}
		
	Step (ii): We then assume that (\ref{Hk}) holds for some $k\geq 0$, and proceed to prove that it holds also for $k+1$. 
	First of all, taking the $ \nu $-th derivative w.r.t.~$v$ on both sides of \eqref{PFS} gives	
	\begin{equation} \label{fmu}
	\partial_t(\partial^{\nu}_v\fe) = \partial^{\nu}_v \mP_NQ^{R}(\fe,\fe)= \mP_N \partial^{\nu}_vQ^{R}(\fe,\fe).
	\end{equation}
       Multiplying (\ref{fmu}) by $ \partial^{\nu}_v\fe $ and integrating over $\D$ then yields
	\begin{equation}\label{H1}
	\frac{1}{2} \frac{\rd}{\rd t}\left\|\partial^{\nu}_v\fe\right\|^{2}_{L^{2}(\D)} =  \int_{\D} \mP_N\partial^{\nu}_v Q^{R}(\fe,\fe) \partial^{\nu}_v\fe \, \rd v \leq  \left\|\partial^{\nu}_v  Q^{R}(\fe,\fe)\right\|_{L^{2}(\D)} \left\|\partial^{\nu}_v\fe\right\|_{L^{2}(\D)}.
	\end{equation}	
	By adding \eqref{H1} with $|\nu|\leq k+1$ altogether and using the Cauchy-Schwarz inequality, we find that 
	\begin{equation}
	\frac{1}{2} \frac{\rd}{\rd t}\left\|\fe\right\|^{2}_{H^{k+1}(\D)} \leq \left\|Q^{R}(\fe,\fe)\right\|_{H^{k+1}(\D)} \left\|\fe\right\|_{H^{k+1}(\D)},
	\end{equation}
	i.e.,
	\begin{equation}\label{H2}
	 \frac{\rd}{\rd t}\left\|\fe\right\|_{H^{k+1}(\D)} \leq \left\|Q^{R}(\fe,\fe)\right\|_{H^{k+1}(\D)}.
	\end{equation}
	On the other hand,
\begin{equation} 
	\begin{split}
	&\left\|Q^{R}(\fe,\fe)\right\|^{2}_{H^{k+1}(\D)} = \left\|Q^{R}(\fe,\fe)\right\|^{2}_{H^{k}(\D)} + \sum_{|\nu|= k+1}  \left\|\partial_v^\nu Q^{R}(\fe,\fe)\right\|^{2}_{L^{2}(\D)}\\
	=&\left\|Q^{R}(\fe,\fe)\right\|^{2}_{H^{k}(\D)} + \sum_{|\nu|= k+1} \left\| \sum_{\mu\leq \nu} \binom{\nu}{\mu} Q^{R}(\partial_v^{\mu}\fe,\partial_v^{\nu-\mu}\fe)\right\|^{2}_{L^{2}(\D)}\\
	\leq & \left\|Q^{R}(\fe,\fe)\right\|^{2}_{H^{k}(\D)} + \sum_{|\nu|= k+1} C_0^2\sum_{\mu\leq \nu} \left\|Q^{R}(\partial^{\mu}_v\fe,\partial^{\nu-\mu}_v\fe)\right\|^{2}_{L^{2}(\D)}\\
	= & \left\|Q^{R}(\fe,\fe)\right\|^{2}_{H^{k}(\D)} + \sum_{|\nu|= k+1} C_0^2 \left(\sum_{0<\mu< \nu} \left\|Q^{R}(\partial^{\mu}_v\fe,\partial^{\nu-\mu}_v\fe)\right\|^{2}_{L^{2}(\D)}\right.\\
	&\left.+\left\|Q^{R}(\fe,\partial^{\nu}_v\fe)\right\|^{2}_{L^{2}(\D)}+\left\|Q^{R}(\partial^{\nu}_v\fe,\fe)\right\|^{2}_{L^{2}(\D)}\right)\\
	\leq &C_1^2\left\|\fe\right\|_{H^{k}(\D)}^2 + \sum_{|\nu|= k+1} C_0^2\left(\sum_{0<\mu< \nu}C_2^2 \|\partial^{\mu}_v\fe\|_{L^2(\D)}^2\|\partial^{\nu-\mu}_v\fe\|_{L^2(\D)}^2\right.\\
	&\left.+C_3^2\|f_N\|_{L^1(\D)}^2\left\|\partial_v^{\nu}\fe\right\|_{L^2(\D)}^2+C_4^2\|\partial_v^\nu \fe\|_{L^1(\D)}^2\left\|\fe\right\|_{L^2(\D)}^2\right)\\
	\leq & C_5^2\left\|\fe\right\|_{H^{k}(\D)}^2+C_6^2(\|f_N\|_{L^1(\D)}^2+\|f_N\|_{L^2(\D)}^2)\left\|\fe\right\|_{H^{k+1}(\D)}^2\\
	\leq & C_5^2K_k^2+C_6^2(M^2+K_0^2)\left\|\fe\right\|_{H^{k+1}(\D)}^2,
	\end{split}
	\end{equation}
	where in the third last inequality, we used (\ref{QHk}) in the first line and (\ref{QLp}) in the second line. In the last inequality, we used the induction hypothesis.
	
	Then (\ref{H2}) becomes		
	\begin{equation}
	\frac{\rd}{\rd t}\left\|\fe\right\|_{H^{k+1}(\D)} \leq C_6(M+K_0)\left\|\fe\right\|_{H^{k+1}(\D)} + C_5K_k.
	\end{equation}
	By the Gr{\"o}nwall's inequality, we have 
	\begin{equation} \label{QQ}
	\begin{split}
	\left\|\fe(t)\right\|_{H^{k+1}(\D)} \leq &  \e^{C_6 (M+K_0) t_0} \left( \left\|f_{N}^0\right\|_{H^{k+1}(\D)} + \frac{C_5K_k}{C_6(M+K_0)} \right):= K_{k+1}, \quad \forall t\in [0,t_0].
	\end{split}
	\end{equation}
	This completes the induction argument for $k+1$.	
	
	In particular, the explicit formula of $K_0$ is given in (\ref{priori}) and the formula of $K_1$ is implied by (\ref{QQ}) when $k=0$.			
\end{proof}

We now proceed to estimate the negative part of the solution, which relies on a careful estimate of both gain and loss terms of the collision operator. This estimate will play a key role in the main theorem.

\begin{proposition}\label{regularity1}
	Let the collision kernel $ B $ and truncation parameters $R$ and $L$ satisfy the assumptions (\ref{RL}), (\ref{kinetic}), (\ref{angular}), and (\ref{cutoff}). For the numerical system \eqref{PFS}, assume that the initial condition $ f_{N}^0\in H^1(\D) $ and that the solution $ \fe $ has a $ L^{1}$ bound up to some time $t_0$:
	\begin{equation}\label{fNL2}
	\forall t\in [0, t_0], \quad \left\| \fe(t) \right\|_{L^{1}(\D)} \leq  M,
	\end{equation}
	then 
	\begin{equation} \label{K1}
	\forall t\in [0, t_0], \quad \left\|\fe(t)\right\|_{L^2(\D)} \leq  K_0, \quad \left\|\fe(t)\right\|_{H^{1}(\D)} \leq  K_1,
\end{equation}
and $f^-_N$, the negative part of $f_N$, satisfies
\begin{equation} \label{fneg}
	\forall t\in [0, t_0], \quad \left\|\fe^{-}(t)\right\|_{L^{2}(\D)} \leq \e^{t_0D_3(M+K_0) } \left(\left\|f_{N}^{0,-}\right\|_{L^{2}(\D)} +\frac{D_4 K_1^2}{MN}\right),
\end{equation}
where $K_0$, $K_1$ are given in (\ref{KK1}), and $D_3$ and $D_4$ are constants depending only on the truncation parameters $R$, $L$, dimension $d$, and the collision kernel $B$.
\end{proposition}

\begin{proof}
First of all, since $ f_{N}^0\in H^1(\D) $, Proposition~\ref{regularity} (when $k=1$) directly yields (\ref{K1}).

Equipped with this regularity, we now estimate the negative part of $f_N$. Note that $f_N=f_N^+ - f^-_N$, $|f_N|=f^+_N + f^-_N$. We first rewrite (\ref{PFS}) as
\begin{equation} \label{PFS1}
	\partial_{t} f_N = Q^{R,+}(f_N,f_N) - Q^{R,-}(f_N,f_N) + E_{N}(f_N),
\end{equation}
with 
\begin{equation}
E_N(f_N):=\mP_NQ^R(f_N,f_N)-Q^R(f_N,f_N).
\end{equation}

For the gain term, we have
\begin{equation}
	\begin{split}
	Q^{R,+}(\fe,\fe) \fe \mathbf{1}_{\left\lbrace \fe\leq 0\right\rbrace } = &  Q^{R,+}(\fe^{+} - \fe^{-}, \fe^{+} - \fe^{-}) \fe \mathbf{1}_{\left\lbrace \fe\leq 0\right\rbrace }\\
	= & \left[ Q^{R,+}(\fe^{+}, \fe^{+}) - Q^{R,+}(\fe^{+}, \fe^{-}) - Q^{R,+}(\fe^{-}, \fe^{+}) + Q^{R,+}(\fe^{-}, \fe^{-}) \right] \fe \mathbf{1}_{\left\lbrace \fe\leq 0\right\rbrace }\\
	= & \left[ -Q^{R,+}(\fe^{+}, \fe^{+}) + Q^{R,+}(\fe^{+}, \fe^{-}) +Q^{R,+}(\fe^{-}, \fe^{+}) - Q^{R,+}(\fe^{-}, \fe^{-}) \right] \fe^-\\
	\leq & \left[ Q^{R,+}(\fe^{+}, \fe^{-}) + Q^{R,+}(\fe^{-}, \fe^{+}) \right] \fe^{-}.
	\end{split}
\end{equation}
Hence
\begin{equation} \label{Q+}
	\begin{split}
	\int_{\D} Q^{R,+}(\fe,\fe) \fe \mathbf{1}_{\left\lbrace \fe\leq 0\right\rbrace } \, \rd v &\leq  \int_{\D} \left[  Q^{R,+}(\fe^{+}, \fe^{-}) + Q^{R,+}(\fe^{-}, \fe^{+}) \right] \fe^{-} \, \rd v \\
	&\leq  \left\|Q^{R,+}(\fe^{+}, \fe^{-})+Q^{R,+}(\fe^{-}, \fe^{+})\right\|_{L^{2}(\D)} \left\|\fe^{-}\right\|_{L^{2}(\D)}\\
	& \leq C_0\left\| \fe^+\right\|_{L^{1}(\D)} \left\| \fe^{-} \right\|^{2}_{L^{2}(\D)}+C_0\left\| \fe^-\right\|_{L^{1}(\D)}\left\| \fe^+\right\|_{L^{2}(\D)} \left\| \fe^{-} \right\|_{L^{2}(\D)}\\
	& \leq C_0\left\| \fe\right\|_{L^{1}(\D)} \left\| \fe^{-} \right\|^{2}_{L^{2}(\D)}+C_0'\left\| \fe\right\|_{L^{2}(\D)} \left\| \fe^{-} \right\|^2_{L^{2}(\D)},
	\end{split}
\end{equation}
where we used the estimate (\ref{QGLp1}) for the gain term.

For the loss term, we have
\begin{equation}
-Q^{R,-}(\fe,\fe) \fe \mathbf{1}_{\left\lbrace \fe\leq 0\right\rbrace } = -  L^R(\fe)\fe \fe\mathbf{1}_{\left\lbrace \fe\leq 0\right\rbrace }=-L^R(\fe)\fe^-\fe^-=-Q^{R,-}(\fe,\fe^-)\fe^-,
\end{equation}
where we used the structure of the loss term, see (\ref{QLS}).
Hence
\begin{equation} \label{Q-}
\begin{split}
-\int_{\D} Q^{R,-}(\fe,\fe) \fe \mathbf{1}_{\left\lbrace \fe\leq 0\right\rbrace } \, \rd v &= -\int_{\D} Q^{R,-}(\fe,\fe^-)\fe^- \, \rd v\\
& \leq \|Q^{R,-}(\fe,\fe^-) \|_{L^2(\D)}\|\fe^-\|_{L^2(\D)}\\
&\leq C_1 \|\fe\|_{L^1(\D)} \left\| \fe^{-} \right\|^{2}_{L^{2}(\D)},
\end{split}
\end{equation}		
where we used the estimate (\ref{QLLp}) for the loss term.

For the remainder $E_N$, we have
\begin{equation} 
\begin{split}
\left\|E_{N}(\fe)\right\|_{L^{2}(\D)}&=\|\mP_NQ^R(f_N,f_N)-Q^R(f_N,f_N)\|_{L^{2}(\D)}\\
&\leq \frac{C_2}{N}\|Q^R(f_N,f_N)\|_{H^1(\D)}\\
&\leq \frac{C_2}{N}\|f_N\|_{H^1(\D)}^2,
\end{split}
\end{equation}
where we used the well-known property of the projection operator and estimate (\ref{QHk}). Hence
\begin{equation} \label{EN}
\begin{split}
	\int_{\D} E_{N}(\fe) \fe \mathbf{1}_{\left\lbrace \fe\leq 0\right\rbrace } \, \rd v & =    -\int_{\D} E_{N}(\fe) \fe^-\, \rd v       \\
	& \leq \left\|E_{N}(\fe)\right\|_{L^{2}(\D)} \left\|\fe^{-}\right\|_{L^{2}(\D)} \\
	& \leq \frac{C_2}{N}\|\fe\|_{H^1(\D)}^2 \left\|\fe^{-}\right\|_{L^{2}(\D)}.
\end{split}
\end{equation}

For the left hand side, we have
\begin{equation}
 \fe \mathbf{1}_{\left\lbrace \fe\leq 0\right\rbrace }\partial_{t} f_N =-\fe^- \partial_t (f_N^+-f_N^-)=-\fe^- (\mathbf{1}_{\left\lbrace \fe \geq 0\right\rbrace }\partial_tf_N-\partial_t f_N^-)=\fe^-\partial_t f_N^-.
\end{equation}

Therefore, multiplying $ \fe \mathbf{1}_{\left\lbrace \fe\leq 0\right\rbrace } $ to both hand sides of \eqref{PFS1} and integrating over $ \D $, together with \eqref{Q+}, \eqref{Q-} and \eqref{EN}, yields
	\begin{equation}
	\frac{1}{2} \frac{\rd}{\rd t}\|\fe^{-}\|^{2}_{L^{2}(\D)} \leq  \left[(C_0+C_1)\left\|\fe \right\|_{L^{1}(\D)}+C_0'\left\| \fe\right\|_{L^{2}(\D)} \right]\left\|\fe^{-}\right\|^{2}_{L^{2}(\D)} +\frac{C_2}{N}\|\fe \|_{H^1(\D)}^2 \left\|\fe^{-}\right\|_{L^{2}(\D)},
	\end{equation}
i.e.,
\begin{equation}
\begin{split}
	\frac{\rd}{\rd t}\|\fe^{-}\|_{L^{2}(\D)} &\leq   \left[(C_0+C_1)\left\|\fe \right\|_{L^{1}(\D)}+C_0'\left\| \fe\right\|_{L^{2}(\D)} \right] \left\|\fe^{-}\right\|_{L^{2}(\D)} +\frac{C_2}{N}\|\fe\|_{H^1(\D)}^2\\
	&\leq  \left[(C_0+C_1) M+C_0'K_0 \right] \left\|\fe^{-} \right\|_{L^{2}(\D)} +\frac{C_2 K_1^2}{N},
\end{split}
\end{equation}
where we have taken into account the $L^1$ bound and $L^2$, $H^1$ bounds of $f_N$ obtained earlier. By the Gr{\"o}nwall's inequality, we finally obtain the desired estimate (\ref{fneg}).
\end{proof}

\subsection{Local well-posedness of the solution $f_N$ on a small time interval $ [t_0,t_0+\tau]$}

To prepare for the main theorem, we establish a local existence and uniqueness result and some stability bounds of the solution.

\begin{proposition}\label{localexistence}
Let the collision kernel $ B $ and truncation parameters $R$ and $L$ satisfy the assumptions (\ref{RL}), (\ref{kinetic}), (\ref{angular}), and (\ref{cutoff}). Assume that the initial condition $f^0(v)$ to the original problem (\ref{ABE}) belongs to $L^{1}\cap L^2(\D)$ and define
	\begin{equation}
	M_{f^0,1}=\|f^0\|_{L^1(\D)},  \quad  M_{f^0,2}=\left\|f^{0}\right\|_{L^2(\D)}.
	\end{equation}
For the numerical system (\ref{PFS}), assume that we evolve it starting at a certain time $t_0$ and the initial condition satisfies
	\begin{equation}
	\|f_N(t_0)\|_{L^1(\D)}\leq 2 M_{f^0,1}, \quad \|f_N(t_0)\|_{L^2(\D)}\leq \e^{2D_0M_{f^0,1}T}M_{f^0,2},
	\end{equation}
then there exists a local time $ \tau$ such that \eqref{PFS} admits a unique solution $ \fe = \fe(t,\cdot) \in L^{1}\cap L^{2}(\D) $ on $ [t_0,t_0+\tau]$. In particular, one can choose
\begin{equation} \label{tau}
\tau=\frac{1}{2(D_5M_2+D_6M_1)}, \quad \text{with} \quad M_{1} = 4M_{f^0,1},  \quad M_{2} = 2\e^{2D_0M_{f^0,1}T}M_{f^0,2},
\end{equation}
such that
\begin{equation}
\forall t\in [t_0,t_0+\tau], \quad  \|f_N(t)\|_{L^1(\D)} \leq M_{1}, \quad  \|f_N(t)\|_{L^2(\D)} \leq M_{2},
\end{equation}
where $T$ is the final prescribed time, $D_0$ is the constant appearing in (\ref{KK1}), and $D_5$, $D_6$ are constants depending only on the truncation parameters $R$, $L$, dimension $d$, and the collision kernel $B$.
\end{proposition}

\begin{proof}	
	We construct the solution by a fixed point argument. 
	
	Given $ M_{1}, M_{2} > 0$ and small enough time $ \tau >0 $ to be specified later, we define the space $\chi$ by
\begin{equation}
	\chi= \left\lbrace  f\in L^{\infty}([t_0,t_0+\tau]; L^1\cap L^2(\D)): \sup\limits_{t\in [t_0,t_0+\tau]}\left\|f(t, \cdot)\right\|_{L^1(\D)} \leq M_{1}, \sup\limits_{t\in [t_0,t_0+\tau]}\left\|f(t, \cdot)\right\|_{L^2(\D)} \leq M_{2} \right\rbrace,
	\end{equation}
which is a complete metric space with respect to the induced distance
	\begin{equation}\label{distance}
	d(f, \tilde{f}) : = \left\| f - \tilde{f} \right\|_{\chi} =\sup\limits_{t\in [t_0,t_0+\tau]}  \left\| f(t, \cdot) - \tilde{f}(t, \cdot) \right\|_{L^2(\D)}.
	\end{equation}

	For any $f_N\in \chi$, we define the operator $ \Phi $ as
	\begin{equation}
	\Phi(f_{N})(t, v) = f_N(t_0,v) + \int_{t_0}^{t}\mP_N Q^R(f_N,f_N)(s,v)\, \rd s, \quad \forall t\in[t_0,t_0+\tau].
	\end{equation}
	We proceed to show that the mapping $\Phi$ has a unique fixed point in $\chi$. 
	 			
\vspace{0.1in}
	Step (i): We first show that $ \Phi $ maps $ \chi$ into itself: $\Phi(\chi)\subset \chi$. For any $ \fe\in \chi $ and $t\in [t_0,t_0+\tau]$,
	\begin{equation} 
	\begin{split}
	\left\| \Phi(\fe)(t, \cdot)\right\|_{L^1(\D)} \leq & \left\| f_{N}(t_0) \right\|_{L^1(\D)} + \int_{t_0}^{t} \left\|  \mP_NQ^{R}(\fe,\fe)(s, \cdot)\right\|_{L^1(\D)}\rd s \\
	\leq & \left\| f_{N}(t_0)  \right\|_{L^1(\D)}  + \tau (2L)^{d/2} \sup\limits_{t\in [t_0,t_0+\tau]} \left\|\mP_NQ^R(\fe,\fe)(t, \cdot)\right\|_{L^2(\D)} \\
	\leq & \left\| f_{N}(t_0)  \right\|_{L^1(\D)}  +  \tau C_{R,L,d,2}(B) (2L)^{d/2} \sup\limits_{t\in [t_0,t_0+\tau]}\left( \left\| \fe(t, \cdot) \right\|_{L^1(\D)} \left\| \fe(t, \cdot) \right\|_{L^2(\D)}\right) \\
	\leq & \left\|f_{N}(t_0)  \right\|_{L^1(\D)}  + \tau C_{R,L,d,2}(B) (2L)^{d/2} M_{1} M_{2}, 
	\end{split}
	\end{equation}
	where we used (\ref{QLp}). Similarly,
	\begin{equation} 
	\begin{split}
	\left\| \Phi(\fe)(t, \cdot) \right\|_{L^2(\D)} \leq & \left\| f_{N}(t_0)\right\|_{L^2(\D)} + \int_{t_0}^{t} \left\|\mP_N Q^{R}(\fe,\fe)(s, \cdot)\right\|_{L^2(\D)}\rd s\\
	\leq & \left\| f_{N}(t_0) \right\|_{L^2(\D)}  + \tau \sup\limits_{t\in [t_0,t_0+\tau]}\left\|\mP_NQ^R(\fe,\fe)(t, \cdot)\right\|_{L^2(\D)}\\
	\leq & \left\| f_{N}(t_0) \right\|_{L^2(\D)}  + \tau C_{R,L,d,2}(B) \sup\limits_{t\in [t_0,t_0+\tau]}\left( \left\| \fe(t, \cdot) \right\|_{L^1(\D)} \left\| \fe(t, \cdot) \right\|_{L^2(\D)}\right)\\
	\leq & \left\| f_{N}(t_0) \right\|_{L^2(\D)}  + \tau C_{R,L,d,2}(B) M_1 M_{2}.
	\end{split}
	\end{equation}
	
	Step (ii): We next show that $ \Phi $ is a contraction mapping on $ \chi $. For any $f_{N}, \tilde{f}_{N}\in\chi$ with the same initial datum $ f_{N}(t_0)$, we have
	\begin{equation} 
	\begin{split}
	\left\|\Phi(f_{N})-\Phi(\tilde{f}_{N})\right\|_{\chi}=& \sup\limits_{t\in [t_0,t_0+\tau]} \left\|\Phi(f_{N})(t, \cdot)-\Phi(\tilde{f}_{N})(t, \cdot)\right\|_{L^2(\D)}\\
	\leq & \sup\limits_{t\in [t_0,t_0+\tau]}\int_{t_0}^{t}  \left\|\mP_NQ^R(\fe,\fe)(s, \cdot)-\mP_NQ^R(\tilde{f}_{N},\tilde{f}_{N})(s, \cdot)\right\|_{L^2(\D)}\, \rd s \\
	\leq & \tau \sup\limits_{t\in [t_0,t_0+\tau]}  \left\|Q^R(\fe,\fe)(t, \cdot)-Q^R(\tilde{f}_{N},\tilde{f}_{N})(t, \cdot)\right\|_{L^2(\D)} \\
	\leq & \tau \sup\limits_{t\in [t_0,t_0+\tau]} \left( \left\|Q^R(\fe-\tilde{f}_{N},\fe)(t, \cdot)\right\|_{L^2(\D)}+\left\|Q^R(\tilde{f}_{N},\fe-\tilde{f}_{N})(t, \cdot)\right\|_{L^2(\D)}\right )\\
	\leq & \tau C_{R,L,d,2}(B) \sup\limits_{t\in [t_0,t_0+\tau]} \left( \left\|\fe-\tilde{f}_{N}\right\|_{L^1(\D)}\|f_N\|_{L^2(\D)}+\left\|\fe-\tilde{f}_{N}\right\|_{L^2(\D)}\|\tilde{f}_N\|_{L^1(\D)} \right )\\
	\leq & \tau C_{R,L,d,2}(B)((2L)^{d/2}M_2+M_1) \left(\sup\limits_{t\in [t_0,t_0+\tau]}  \left\|\fe(t, \cdot)-\tilde{f}_{N}(t, \cdot)\right\|_{L^2(\D)}\right)\\
	\leq &  \tau(C_{R,L,d,2}(B)(2L)^{d/2}M_2+C_{R,L,d,2}(B) M_1)\left\|\fe-\tilde{f}_{N}\right\|_{\chi}.
	\end{split}
	\end{equation}

Therefore, if we define $D_5=C_{R,L,d,2}(B)(2L)^{d/2}$, $D_6=C_{R,L,d,2}(B)$, and choose $M_1$, $M_2$ and $\tau$ as given in (\ref{tau}), we would have
\begin{equation} 
\left\| f_{N}(t_0) \right\|_{L^1} + \tau D_5 M_{1} M_{2} \leq M_{1}, \quad \left\|f_{N}(t_0) \right\|_{L^2} + \tau D_6 M_1M_2 \leq M_{2}, \quad  \tau(D_5 M_2+D_6M_1)< 1.
\end{equation}
So $\Phi: \chi \rightarrow \chi$ is a contraction mapping. According to the Banach fixed point theorem, \eqref{PFS} admits a unique solution on $[t_0,t_0+\tau]$.
\end{proof}

%%%%%%%%%%%%%%%%%%%%%%%%%%%%%%%%%%%%%%%%%%%%%%%%%%%%%%%%%
\subsection{Well-posedness and stability of the solution $f_N$ on an arbitrary bounded time interval $ [0,T] $}
\label{subsec:main}

We are ready to present our main result.

\begin{theorem}\label{existencetheorem}
	Let the collision kernel $ B $ and truncation parameters $R$ and $L$ satisfy the assumptions (\ref{RL}), (\ref{kinetic}), (\ref{angular}), and (\ref{cutoff}). Let the initial condition $f^{0}(v)$ to the original problem (\ref{ABE}) and the numerical solution $f_N^0(v)$ to the numerical system (\ref{PFS}) satisfy the assumptions specified in Section~\ref{subsec:initial}, i.e., $f^0(v)$ is periodic, non-negative, and belongs to $L^1\cap H^1(\D)$, $f_N^0$ satisfies (\ref{con(a)})--(\ref{con(d)}).
	Define 
	\begin{equation}
	M_{f^0,1}=\|f^0\|_{L^1(\D)},  \quad  M_{f^0,2}=\left\|f^{0}\right\|_{L^2(\D)}.
	\end{equation}
	Then there exists an integer $N_0$ depending on the final time $T$ and initial condition $f^0$,
such that for all $ N>N_{0} $, the numerical system (\ref{PFS}) admits a unique solution $ \fe = \fe(t,\cdot) \in L^{1} \cap H^1(\D) $ on the time interval $ [0,T] $. Furthermore, for all $ N>N_{0} $, $f_N$ satisfies the following stability estimates
	\begin{equation}\label{fNL1L2}
	\forall t\in [0,T], \quad \left\| f_{N}(t) \right\|_{L^{1}(\D)} \leq 2M_{f^0,1}, \quad \left\| f_{N}(t) \right\|_{L^{2}(\D)} \leq \e^{2D_0M_{f^0,1}T}M_{f^0,2},
	\end{equation}
	where $D_0$ is the constant appearing in (\ref{KK1}).
\end{theorem}

\begin{proof}
The proof is based on iteration.	Given $T$, $M_{f^0,1}$, and $M_{f^0,2}$, we first choose $\tau$ according to (\ref{tau}). Then we define $t=0, \tau, 2\tau, \dots, n\tau, \dots$ until we cover the final time $T$. WLOG, we assume $T$ is some integral multiple of $\tau$.

Step (i): At initial time $t=0$, we first choose $N$ such that
	\begin{equation} \label{initial}
	\|f^0_N\|_{L^1(\D)}\leq 2M_{f^0,1},
	\end{equation}
which is possible due to the condition (\ref{con(c)}). Also we have $\|f^0_N\|_{L^2(\D)}\leq \|f^0\|_{L^2(\D)} \leq e^{2D_0M_{f^0,1}T}M_{f^0,2}$ due to the condition (\ref{con(b)}). Then by Proposition~\ref{localexistence}, there exists a unique solution $f_N(t,\cdot)\in L^1\cap L^2(\D)$ over the time interval $[0,\tau]$ and
\begin{equation}
\forall t\in [0,\tau], \quad \|f_N(t)\|_{L^1(\D)}\leq 4M_{f^0,1}.
\end{equation}
Using this $L^1$ bound and that $f_N^0\in {H^1(\D)}$ (due to (\ref{con(b)})), we can invoke the Proposition~\ref{regularity1} to derive that
\begin{equation}
\forall t\in [0, \tau], \quad \|f_N(t)\|_{L^2(\D)}\leq K_0(\tau), \quad \|f_N(t)\|_{H^1(\D)}\leq K_1(\tau),
\end{equation}
and
\begin{equation}
\forall t\in [0, \tau], \quad \left\|\fe^{-}(t)\right\|_{L^{2}(\D)} \leq \e^{\tau D_3(4M_{f^0,1}+K_0(\tau)) } \left(\left\|f_{N}^{0,-}\right\|_{L^{2}(\D)} +\frac{D_4 K_1^2(\tau)}{4M_{f^0,1}N}\right), \label{fN-}
\end{equation}
with
\begin{equation}
K_0(\tau):=\e^{\tau D_0 4M_{f^0,1}}M_{f^0,2},   \quad K_1(\tau):=\e^{\tau D_1 \left(4M_{f^0,1}+K_0(\tau) \right)} \left( \left\|f^0\right\|_{H^{1}(\D)}+D_2\right).
\end{equation}
Note that we relaxed the bounds $K_0$, $K_1$ a bit (so that they depend only on $f^0$ but not $f_N^0$) using the condition (\ref{con(b)}) again.

On the other hand, noticing that $|f_N| = 2f_N^- + f_N$, we have
\begin{equation}
\begin{split}
\|f_N(t)\|_{L^1(\D)}&=\int_{\D}|f_N(t,v)|\,\rd{v}=2\int_{\D}f_N^-(t,v)\,\rd{v}+\int_{\D}f_N(t,v)\,\rd{v}\\
&=2\|f_N^-(t)\|_{L^1(\D)}+\int_{\D}f^0(v)\,\rd{v}\\
& \leq 2(2L)^{d/2}\|f_N^-(t)\|_{L^2(\D)}+M_{f^0,1},
\end{split}
\end{equation}
where we used the important mass conservation property in Lemma~\ref{lemma:conv} and (\ref{con(a)}) in the second line.

Therefore, if we can control $\|f_N^-(t)\|_{L^2(\D)}$, then $\|f_N(t)\|_{L^1(\D)}$ will be controlled. Thanks to the estimate (\ref{fN-}), we can simply choose $N$ large enough such that the following is satisfied:
\begin{equation} \label{N0}
\mathcal{K}:=\e^{T D_3(4M_{f^0,1}+K_0(T)) } \left(\left\|f_{N}^{0,-}\right\|_{L^{2}(\D)} +\frac{D_4 K_1^2(T)}{4M_{f^0,1}N}\right) \leq \frac{M_{f^0,1}}{2(2L)^{d/2}},
\end{equation}
then we have
\begin{equation} \label{fNL1L21}
\forall t\in [0,\tau], \quad \|f_N(t)\|_{L^1(\D)}\leq 2M_{f^0,1}.
\end{equation}
Note that (\ref{N0}) is possible due to the condition (\ref{con(d)}). Also, it is easy to see that the quantity $\mathcal{K}$ is an increasing function in time. Hence if $T$ in (\ref{N0}) is replaced by some $t_0\leq T$, (\ref{N0}) still holds.

Combining the above choice of $N$ with the one at the beginning to satisfy (\ref{initial}), we have found an integer $N_0$, depending only on the final time $T$ and initial condition $f^0$, such that for all $N> N_0$, (\ref{PFS}) admits a unique solution $f_N(t,\cdot)\in L^1\cap H^1(\D)$ on $[0,\tau]$ which satisfies (\ref{fNL1L21}). 

Step (ii): Generally at time $t=n\tau$ ($n\geq 1$), we have
\begin{equation}  \label{initial1}
\forall t\in [0,n\tau], \quad f_N(t,\cdot)\in L^1\cap H^1(\D), \quad \|f_N(t)\|_{L^1(\D)}\leq 2M_{f^0,1}.
\end{equation}
Then by Proposition~\ref{regularity} (with $k=0$), we have
\begin{equation} 
\forall t\in [0,n\tau],  \quad \|f_N(t)\|_{L^2(\D)}\leq e^{2D_0M_{f^0,1}n \tau}\|f_N^0\|_{L^2(\D)}\leq e^{2D_0M_{f^0,1}T}M_{f^0,2}.
\end{equation}
Then by Proposition~\ref{localexistence}, there exists a unique solution $f_N(t,\cdot)\in L^1\cap L^2(\D)$ on $[n\tau, (n+1)\tau]$ and
\begin{equation}
\forall t\in [n\tau, (n+1)\tau], \quad \|f_N(t)\|_{L^1(\D)}\leq 4M_{f^0,1}.
\end{equation}
Using this $L^1$ bound and that $f_N^0\in {H^1(\D)}$, we can invoke the Proposition~\ref{regularity1} over the interval $[0,(n+1)\tau]$ to derive that
\begin{equation}
\forall t\in [0, (n+1)\tau], \quad \|f_N(t)\|_{L^2(\D)}\leq K_0((n+1)\tau), \quad \|f_N(t)\|_{H^1(\D)}\leq K_1((n+1)\tau),
\end{equation}
and
\begin{equation}
\forall t\in [0,(n+1)\tau], \quad \left\|\fe^{-}(t)\right\|_{L^{2}(\D)} \leq \e^{(n+1)\tau D_3(4M_{f^0,1}+K_0((n+1)\tau))} \left(\left\|f_{N}^{0,-}\right\|_{L^{2}(\D)} +\frac{D_4 K_1^2((n+1)\tau)}{4M_{f^0,1}N}\right) \leq \mathcal{K},
\end{equation}
i.e., the same choice of $N$ chosen above would still make
\begin{equation} 
\forall t\in [0,(n+1)\tau], \quad \|f_N(t)\|_{L^1(\D)}\leq 2M_{f^0,1}.
\end{equation}
That is, at time $t=(n+1)\tau$, we are back to the situation (\ref{initial1}) at $t=n\tau$.

Repeating Step (ii) until $t=T$, we can show that there exists a unique solution $f_N(t,\cdot)\in L^1\cap H^1(\D)$ on $[0,T]$, and
\begin{equation}
\forall t\in [0,T], \quad \|f_N(t)\|_{L^1(\D)}\leq 2M_{f^0,1}.
\end{equation}
Finally, by Proposition~\ref{regularity} (with $k=0$) again, we obtain
\begin{equation}
\forall t\in [0,T], \quad \left\| f_{N}(t) \right\|_{L^{2}(\D)} \leq e^{2D_0M_{f^0,1}T}M_{f^0,2}.
\end{equation}
\end{proof}

%%%%%%%%%%%%%%%%%%%%%%%%%%%%%%%%%%%%%%%%%%
\section{Convergence and spectral accuracy of the method}
\label{sec:conv}

With the well-posedness and stability of the numerical solution established in the previous section, the convergence of the method is straightforward. 

In this section, we assume that the initial condition $f^{0}(v)$ to the original problem (\ref{ABE}) is periodic, non-negative, and belongs to $L^1\cap H^k(\D)$ for some integer $k\geq 1$. In fact, it has been proved in \cite[Proposition 5.1]{FM11} that there exists a unique global non-negative solution $f(t,\cdot)\in H^k(\D)$. Furthermore, $\|f(t)\|_{H^k(\D)}\leq C_k(f^0), \ \forall t\geq 0$, where $C_k$ is a constant depending only on the initial condition.

For the numerical system (\ref{PFS}), we consider the initial condition $f^0_N=\mP_Nf^0$ for simplicity. According to the discussion in Remark~\ref{rmk}, we further assume that $f^0$ is, say, H\"{o}lder continuous, so that the four conditions (\ref{con(a)})--(\ref{con(d)}) are satisfied. Then by Theorem~\ref{existencetheorem}, there exists a unique solution $f_N(t,\cdot)\in L^1\cap H^1(\D)$ over the time interval $[0,T]$. Furthermore, $\|f_N(t)\|_{L^2(\D)}\leq C_0(T,f^0), \ \forall t \in[0,T]$, where $C_0$ is a constant depending only on the final time $T$ and initial condition $f^0$.

Define the error function 
\begin{equation}
e_{N}(t,v) = \mathcal{P}_{N} f(t,v) - \fe(t,v).
\end{equation}
We can show the following:
\begin{theorem} \label{spectralaccuracy}
Let the collision kernel $ B $ and truncation parameters $R$ and $L$ satisfy the assumptions (\ref{RL}), (\ref{kinetic}), (\ref{angular}), and (\ref{cutoff}). Choose $N_0$ such that it satisfies the condition in Theorem~\ref{existencetheorem}, then the Fourier spectral method is convergent for all $N>N_0$ and exhibits spectral accuracy. In particular, we have
\begin{equation} \label{final}
\forall t\in [0,T], \quad \left\| e_{N}(t) \right\|_{L^2(\D)} \leq \frac{C(T,f^0)}{N^k}, \quad \text{for all } N>N_0,
\end{equation}
where $C$ is a constant depending only on the final time $T$ and initial condition $f^0$.

\end{theorem}
\begin{proof}
	
	We first project the original problem \eqref{ABE} to obtain
	\begin{equation}\label{PUP}
	\left\{
	\begin{array}{lr}  
	\partial_{t} \mathcal{P}_{N}f = \mathcal{P}_{N}Q^{R}(f,f),\\
	\mathcal{P}_{N}f(0,v) =  \mathcal{P}_Nf^0, 
	\end{array}
	\right.
	\end{equation}
	Subtracting \eqref{PFS} from (\ref{PUP}) and noting $f^0_N=\mP_Nf^0$, we have
	\begin{equation}\label{error1}
	\left\{
	\begin{split}
	&\partial_{t} e_{N}= \mathcal{P}_{N}\left( Q^{R}(f,f) - Q^{R}(\fe,\fe)\right),\\
	&e_{N}(0,v) =  0.
	\end{split}
	\right.
	\end{equation}
	Multiplying (\ref{error1}) by $ e_{N}$ and integrating over $ \D $, we have
	\begin{equation}
	\begin{split}
	\frac{1}{2} \frac{\rd}{\rd t} \left\| e_N \right\|^{2}_{L^{2}(\D)}  =& \int_{\D} \mathcal{P}_{N}\left( Q^{R}(f,f) - Q^{R}(\fe,\fe)\right) e_{N}\, \rd v\\
	\leq & \left\|  \mathcal{P}_{N}\left( Q^{R}(f,f) - Q^{R}(\fe,\fe)\right)\right\|_{L^2(\D)} \left\|  e_{N} \right\|_{L^2(\D)},\\
	\Rightarrow  \frac{\rd}{\rd t} \left\| e_N \right\|_{L^{2}(\D)} \leq & \left\| Q^{R}(f,f) - Q^{R}(\fe,\fe)\right\|_{L^2(\D)}. 
	\end{split}
	\end{equation}
	Note that
	\begin{equation}
	\begin{split}
	&\left\| Q^{R}(f,f) - Q^{R}(\fe,\fe) \right\|_{L^2(\D)}\\
	\leq& \left\|  Q^{R}(f-\fe,f) \right\|_{L^2(\D)} + \left\|  Q^{R}(\fe,f-\fe) \right\|_{L^2(\D)}\\
	\leq  &C_1 \left\|  f - \fe \right\|_{L^2(\D)} \left(\left\|  f  \right\|_{L^2(\D)} + \left\| \fe \right\|_{L^2(\D)} \right)\\
	\leq  &C_1(T,f^0) \left\|  f - \fe \right\|_{L^2(\D)}.
	\end{split}
	\end{equation}
Also
	\begin{equation}
	\begin{split}
	\left\|  f - \fe \right\|_{L^2(\D)} \leq & \left\|  f - \mathcal{P}_{N}f \right\|_{L^2(\D)}+\left\|  \mathcal{P}_{N}f- \fe \right\|_{L^2(\D)}\\
	\leq  & \frac{C_2 \|f\|_{H^k(\D)}}{N^{k}} + \left\|e_{N} \right\|_{L^2(\D)}\\
	\leq  & \frac{C_2(f^0)}{N^{k}} + \left\|e_{N} \right\|_{L^2(\D)}.
	\end{split}
	\end{equation}
	Therefore, we have
	\begin{equation}
	\frac{\rd}{\rd t} \left\| e_N \right\|_{L^{2}(\D)} \leq C_1(T,f^0)\left\|e_{N} \right\|_{L^2(\D)}+\frac{C_3(T,f^0)}{N^{k}},
	\end{equation}
	which implies
	\begin{equation}
	\forall t\in [0,T], \quad \left\| e_N(t) \right\|_{L^{2}(\D)} \leq e^{C_1(T,f^0)T}\left(\left\|e_{N}(0) \right\|_{L^2(\D)}+\frac{C_3(T,f^0)}{C_1(T,f^0)N^k}\right).
	\end{equation}
	Since $e_N(0,v)\equiv 0$, we finally obtain the desired result in (\ref{final}).
\end{proof}

%%%%%%%%%%%%%%%%%%%%%%%%%%%%%%%%%%%%%%%
%\section{Conclusion}

\section*{Acknowledgement}

JH is grateful to F. Filbet and R. Alonso for the helpful discussion. JH's research was supported in part by NSF grant DMS-1620250 and NSF CAREER grant DMS-1654152. TY's research was partially supported by General Research Fund of Hong Kong, \#11304419.

%%%%%%%%%%%%%%%%%%%%%%%%%%%%%%%%%%%%%%%

\appendix

\section*{Appendix: proof of estimate (\ref{QGLp1}) for the truncated collision operator $Q^{R,+}$ on a bounded domain}

By duality,
\begin{equation}\label{dual}
\left\|Q^{R,+}(g,f)\right\|_{L^{p}(\D)} = \sup \left\lbrace \int_{\D} Q^{R,+}(g,f)(v) \Psi(v)\,\rd v; \ \left\|\Psi\right\|_{L^{p'}(\D)} \leq 1 \right\rbrace.
\end{equation}
With the pre-post collisional change of variables, namely $ (v,v_{*},\sigma)\rightarrow (v',v_{*}', \frac{v-v_{*}}{|v-v_{*}|}) $, which has a unit Jacobian, we can obtain
\begin{equation}\label{dual1}
\begin{split}
\int_{\D} Q^{R,+}(g,f)(v) \Psi(v)\, \rd v =& \int_{\D} \int_{\mathbb{R}^d} \left( \int_{\bS^{d-1}} \mathbf{1}_{|v-v_{*}|\leq R} \Phi(|v-v_*|) b(\sigma \cdot \widehat{(v-v_*)}) \Psi(v') \,\rd \sigma \right) g(v_{*})f(v)\, \rd v_* \,\rd v\\
=&\int_{\D} \int_{\mathcal{B}_{\sqrt{2}L+R}} \left( \int_{\bS^{d-1}} \mathbf{1}_{|v-v_{*}|\leq R} \Phi(|v-v_*|) b(\sigma \cdot \widehat{(v-v_*)}) \Psi(v') \,\rd \sigma \right) g(v_{*})f(v)\, \rd v_* \,\rd v,
\end{split}
\end{equation}
where the second equality is obtained by noting that $|v_*|\leq |v|+|v-v_*|$ and that $v\in \D$ and $|v-v_*|\leq R$.

Then, we define the linear operator $ S $ by
\begin{equation}
\begin{split}
S\Psi(v) &= \int_{\bS^{d-1}} \mathbf{1}_{|v|\leq R} \Phi(|v|) b(\sigma \cdot \hat{v}) \Psi\left(\frac{v+|v|\sigma}{2}\right) \rd \sigma,
\end{split}
\end{equation}
such that \eqref{dual1} can be written as 
\begin{equation}\label{dual2}
\int_{\D} Q^{R,+}(g,f)(v) \Psi(v) \,\rd v = \int_{{\mathcal{B}_{\sqrt{2}L+R}} }g(v_{*}) \left( \int_{\D}  f(v) (\tau_{v_{*}}S(\tau_{-v_{*}}\Psi))(v) \, \rd v  \right)\rd v_{*},
\end{equation}
where $\tau_hf(v):=f(v-h)$.

We shall study the operator $S$ in $L^1$ and $L^{\infty}$ norms. Denote $v^+=\frac{v+|v|\sigma}{2}$, then we have
\begin{equation}
\left|v^+\right| \leq |v|.
\end{equation}
Then
\begin{equation}
\|S\Psi\|_{L^{\infty}(\D)}\leq  \|b\|_{L^1(\mathbb{S}^{d-1})}\|\mathbf{1}_{|v|\leq R} \Phi(|v|)\|_{L^{\infty}{(\D)}} \|\Psi\|_{L^{\infty}{(\mathcal{B}_{\sqrt{2}L})}}.
\end{equation}
Also
\begin{equation}
\begin{split}
\|S\Psi\|_{L^1(\D)}&\leq  \|\mathbf{1}_{|v|\leq R} \Phi(|v|)\|_{L^{\infty}{(\D)}} \int_{\D}\int_{\bS^{d-1}}  b(\sigma \cdot \hat{v}) \left|\Psi(v^+) \right|\rd \sigma\,\rd{v}\\
&\leq  \|\mathbf{1}_{|v|\leq R} \Phi(|v|)\|_{L^{\infty}{(\D)}} \int_{\mathcal{B}_{\sqrt{2}L}}\int_{\bS^{d-1}}  b(\cos \theta) \left|\Psi\left(v^+\right) \right| \frac{2^{d-1}}{\cos^2\theta/2}\,\rd \sigma\,\rd{v^+}\\
& \leq C\|b\|_{L^1(\mathbb{S}^{d-1})}\|\mathbf{1}_{|v|\leq R} \Phi(|v|)\|_{L^{\infty}{(\D)}}\|\Psi\|_{L^1(\mathcal{B}_{\sqrt{2}L})},
\end{split}
\end{equation}

By the Riesz-Thorin interpolation, we deduce
\begin{equation}
\|S\Psi\|_{L^p(\D)} \leq C_{R,L,d,p'}^+(B)\|\Psi\|_{L^p(\mathcal{B}_{\sqrt{2}L})}, \quad 1\leq p \leq \infty,
\end{equation}
where $C_{R,L,d,p'}^+(B)=C^{1/p'}\|b\|_{L^1(\mathbb{S}^{d-1})}\|\mathbf{1}_{|v|\leq R} \Phi(|v|)\|_{L^{\infty}{(\D)}}$. Using this inequality in \eqref{dual2}, we have
\begin{equation}
\begin{split}
\left| \int_{\D} Q^{R,+}(g,f)(v) \Psi(v) \,\rd v \right| \leq &\int_{\mathcal{B}_{\sqrt{2}L+R}}  |g(v_{*})|  \left( \int_{\D}\left|f(v)\right| \left|(\tau_{v_{*}}S(\tau_{-v_{*}}\Psi))(v)\right| \rd v  \right) \rd v_{*}\\
\leq  &\int_{\mathcal{B}_{\sqrt{2}L+R}}\left|g(v_{*})\right| \left\|f\right\|_{L^{p}(\D)} \left\|\tau_{v_{*}}S(\tau_{-v_{*}}\Psi)\right\|_{L^{p'}(\D)} \rd v_{*}\\
\leq  &\int_{\mathcal{B}_{\sqrt{2}L+R}}\left|g(v_{*})\right| \left\|f\right\|_{L^{p}(\D)} \left\|\tau_{v_{*}}S(\tau_{-v_{*}}\Psi)\right\|_{L^{p'}(\mathbb{R}^d)} \rd v_{*}\\
=  &\int_{\mathcal{B}_{\sqrt{2}L+R}}\left|g(v_{*})\right| \left\|f\right\|_{L^{p}(\D)} \left\|S(\tau_{-v_{*}}\Psi)\right\|_{L^{p'}(\mathbb{R}^d)} \rd v_{*}\\
=  &\int_{\mathcal{B}_{\sqrt{2}L+R}} \left|g(v_{*})\right| \left\|f\right\|_{L^{p}(\D)} \left\|S(\tau_{-v_{*}}\Psi)\right\|_{L^{p'}(\D)} \rd v_{*}\\
\leq  & C_{R,L,d,p}^+(B)\int_{\mathcal{B}_{\sqrt{2}L+R}}\left|g(v_{*})\right| \left\|f\right\|_{L^{p}(\D)} \left\|\tau_{-v_{*}}\Psi\right\|_{L^{p'}(\mathcal{B}_{\sqrt{2}L})} \rd v_{*}\\
\leq  & C_{R,L,d,p}^+(B)\int_{\mathcal{B}_{\sqrt{2}L+R}}\left|g(v_{*})\right| \left\|f\right\|_{L^{p}(\D)} \left\|\Psi\right\|_{L^{p'}(\mathcal{B}_{2\sqrt{2}L+R})} \rd v_{*}\\
= & C_{R,L,d,p}^+(B) \left\|g\right\|_{L^{1}(\mathcal{B}_{\sqrt{2}L+R})} \left\|f\right\|_{L^{p}(\D)} \left\|\Psi\right\|_{L^{p'}(\mathcal{B}_{2\sqrt{2}L+R})} \\
\leq & C_{R,L,d,p}^+(B) \left\|g\right\|_{L^{1}(\D)} \left\|f\right\|_{L^{p}(\D)} \left\|\Psi\right\|_{L^{p'}(\D)}\\
\leq  & C_{R,L,d,p}^+(B) \left\|g\right\|_{L^{1}(\D)} \left\|f\right\|_{L^{p}(\D)},
\end{split}
\end{equation}
where the second equality is obtained by noting $\text{Supp}(S\Psi)\subset \mathcal{B}_R\subset \D$ since $R\leq L$, and the second last line is obtained by noting that both $g$ and $\Psi$ are periodic functions on $\D$.

Hence we proved the estimate (\ref{QGLp1}).

\bibliographystyle{plain}
\bibliography{hu_bibtex}
\end{document}